\documentclass[12pt]{article}
\usepackage{amsmath}
\usepackage{amsfonts}
\usepackage{amsthm}
\usepackage{amssymb}    
\usepackage{tikz}
\usepackage{caption}

\newtheorem{theorem}{Theorem}
\newtheorem{lemma}{Lemma}
\newtheorem{corollary}{Corollary}[theorem]

\newtheorem{definition}{Definition}
\newtheorem{claim}{Claim}

\newtheorem{question}{Question}

\def\beq{\begin{equation}}\def\eeq{\end{equation}}
\def\beqn{\begin{eqnarray}}\def\eeqn{\end{eqnarray}}

\def\qed{\ifhmode\unskip\nobreak\fi\quad\ifmmode\Box\else$\Box$\fi}

\newcommand{\pipe}{\,\,\vert\,\,}

\newcommand{\ep}{\varepsilon}
\newcommand{\ex}{\text{ex}}
\newcommand{\hH}{\mathcal{H}}

\newcommand{\tT}{\mathcal{T}}
\newcommand{\indeg}{d_G^-}
\newcommand{\outdeg}{d_G^+}
\newcommand{\red}{{G_r}}
\newcommand{\blue}{{G_b}}
\newenvironment{case}[1]{\smallskip\noindent{\bf Case #1.}}{\smallskip}

\newcommand{\twographs}[4]{
\begin{minipage}[b]{0.45\textwidth}
\centering
#1
\end{minipage}%
\hfill
\begin{minipage}[b]{0.45\textwidth}
#3
\end{minipage}

\begin{minipage}[t]{0.45\textwidth}
\centering
\captionof{figure}{#2}
\end{minipage}%
\hfill
\begin{minipage}[t]{0.45\textwidth}
\captionof{figure}{#4}
\end{minipage}
}

\newenvironment{myfigure}{
\begin{figure}[!h]
\centering
}{
\end{figure}
}

\newenvironment{graph}{
\begin{tikzpicture}
\tikzstyle{every node}=[draw,circle,fill=white,minimum size=4pt, inner sep=0pt]
\tikzstyle{arrow} = [draw,->,>=latex]
\tikzstyle{redarrow} = [draw,->,>=latex,red]
\tikzstyle{bluearrow} = [draw,->,>=latex,blue]
}{
\end{tikzpicture}
}

\newcommand{\circlegraph}[3] 
{
    \def \n {#1}
    \def \radius {#2}

    \foreach \s in {1,...,\n}
    {
        \node[draw] (#3 \s) at ({360/\n * (\s - 1)}:\radius) {};
    }
}

\newcommand{\fivecycle}[1]
{
    \circlegraph{5}{2cm}{#1}
    \draw (#1 1) -- (#1 2);
    \draw (#1 2) -- (#1 3);
    \draw (#1 3) -- (#1 4);
    \draw (#1 4) -- (#1 5);
    \draw (#1 5) -- (#1 1);
}

\newcommand{\sixcycle}[1]
{
    \circlegraph{6}{2cm}{#1}
    \draw (#1 1) -- (#1 2);
    \draw (#1 2) -- (#1 3);
    \draw (#1 3) -- (#1 4);
    \draw (#1 4) -- (#1 5);
    \draw (#1 5) -- (#1 6);
    \draw (#1 6) -- (#1 1);
}

\title{Good graph hunting\footnote{
This research has been conducted in the framework of the Budapest Semesters Undergraduate Research Experience Program; my advisor was Andr\'as Gy\'arf\'as.
}}

\author{Philip Garrison\\
\small Carnegie Mellon University\\[-0.8ex]
\small \texttt{philipgarrison@cmu.edu}}

\begin{document}
\maketitle
\begin{abstract}
Given graphs\footnote{We consider only finite simple graphs.} $H_1, H_2, \dots, H_k$, the Ramsey number $R(H_1, \dots, H_k)$ is the smallest integer $n$ for which in any coloring of the edges of the complete graph $K_n$ with colors $1,2,\dots,k$, there is some color $i$ with a monochromatic copy of $H_i$. 
We call a tuple $(H_1, \dots, H_k)$ {\em good} if for every $k$-coloring of the edges of an $R(H_1, \dots, H_k)$-chromatic graph, there is some color $i$ with a monochromatic copy of $H_i$.
We call a graph $H$ {\em $k$-good} if the $k$-tuple $(H, H, \dots, H)$ is good, and $H$ is {\em good} if it is $k$-good for every $k$. 
Bialostocki and Gy\'arf\'as proved that matchings are good and asked whether every acyclic $H$ is good. 
A natural strategy shows that $P_4$ is $k$-good for $k \not = 3$ and that $(P_4, P_5)$ is good.
We develop a new technique for showing that a graph is $2$-good, and we apply it successfully to $P_5$, $P_6$, and $P_7$.
\end{abstract}

\section{Introduction} 
\label{sec:intro}
Given graphs $H_1, H_2, \dots, H_k$, the Ramsey number $R(H_1, \dots, H_k)$ is the smallest integer $n$ for which in any coloring of the edges of the complete graph $K_n$ with colors $1,2,\dots,k$, there is some color $i$ with a monochromatic copy of $H_i$. 
Bialostocki and Gy\'arf\'as \cite{BGY} raised the following question:
what is the smallest $n$ such that every $n$-chromatic graph $G$ (rather than just $K_n$) has this property.
They noted that such a smallest $n$ exists only for acyclic graphs, and asked whether this $n$ is equal to $R(H_1,\dots,H_k)$ for all acyclic graphs.

\begin{definition}
We call a tuple $(H_1, \dots, H_k)$ {\em good} if for every $k$-coloring of the edges of an $R(H_1, \dots, H_k)$-chromatic graph, there is some color $i$ with a monochromatic copy of $H_i$.\footnote{
$n := R(H_1,\dots,H_k)$ is the {\em smallest} integer with this property, because there is an $n-1$-chromatic graph whose edges have a $k$-coloring such that in every color $i$, there is no monochromatic $H_i$ - namely, $K_{n-1}$.} 
When $H_1 = H_2 = \dots = H_k = H$, we write $R_k(H) = R(H_1, \dots, H_k)$.
We call a graph {\em $k$-good} if every $R_k(H)$-chromatic graph contains a monochromatic $H$ in every $k$-coloring of its edges, and $H$ is {\em good} if it is $k$-good for every $k$. 
\end{definition}

Bialostocki and Gy\'arf\'as proved that matchings are good, extending a classical result of Cockayne and Lorimer \cite{CL}.
As an introduction, we apply Brooks' theorem to show that stars are good graphs.
In Section \ref{sec:2}, we consider the Tur\'an number $\ex(n,H)$ of a graph, which is the maximum number of edges among $H$-free graphs on $n$ vertices.
From this, we get an upper-bound on the number of edges in a graph $G$ which has an edge coloring such that no color has a monochromatic $H$.
For certain graphs $H$, we use this bound to show that such a graph $G$ does not have enough edges to have chromatic number $R_k(H)$.
With this, we prove that $P_4$ is $k$-good for all $k$, with the possible exception of $k = 3$, and show that $(P_4, P_5)$ is good. 

This method is not strong enough to give results for other paths, so we turn to a more sophisticated technique in Section \ref{sec:3}.
We consider the possible families of $H$-free graphs and orient each by orienting away from those vertices with unbounded degree.
Using this technique, we prove Theorem \ref{maintheorem}, which shows that a graph $H$ is $2$-good if the $H$-free graphs can be given a suitable partial orientation.
To prove the claim ``$P_5$, $P_6$, and $P_7$ are $2$-good'' to be true, we demonstrate such partial orientations.
Finally, we conjecture that all $P_N$-free graphs have suitable partial orientations.

\begin{theorem} 
\label{stars}
Let $n \geq 3$, $k \geq 2$. 
Let $G$ be a graph with $\chi(G) \geq k(n-2) + 1$ (where $\chi(G)$ is the chromatic number of $G$).
If $G$ is not the complete graph, then every $k$-coloring of the edges of $G$ contains a monochromatic star on $n$ vertices $S_n$.
\end{theorem}

\begin{proof}
By Brooks' theorem, either (a) $G$ has a vertex of degree at least $k(n-2)+1$, (b) $G$ is an odd cycle and $k(n-2)+1 = 3$, or (c) $G$ is the complete graph on $k(n-2)+1$ vertices.
By assumption, $G$ is not complete, so only (a) or (b) is possible.
\begin{enumerate}
\item[(a)] Let $v$ be a vertex of $G$ with degree at least $k(n-2)+1$.
In any $k$-coloring of the edges of $G$, $v$ must have $n-1$ edges of the same color, so $v$ is the center of a monochromatic $S_n$.

\item[(b)] Suppose $G$ is an odd cycle and $k(n-2)+1 = 3$.
Since $k \geq 2$ and $n \geq 3$ by assumption, we must have $k = 2$ and $n = 3$.
In every $2$ coloring of the edges of an odd cycle, there is a monochromatic $S_3$. \qedhere
\end{enumerate}
\end{proof}

\begin{corollary} The star on $n$ vertices, $S_n$, is $k$-good for all $k \geq 1$.
\end{corollary}

\begin{proof}
Let $G$ be a graph with chromatic number $R_k(S_n)$; we will show that every $k$-coloring of the edges of $G$ has a monochromatic $S_n$.
By definition, if $G$ is the complete graph on $R_k(S_n)$ vertices, every $k$-coloring of the edges of $G$ has a monochromatic $S_n$.
For the case $k=1$, it is easy to see that $R_1(S_n) = n$.
The folklore statement that any $n$-chromatic graph has a copy of every acyclic graph on $n$ vertices shows that, in particular, $S_n$ is $1$-good, so we may assume $k > 1$.
$S_2=K_2$ is trivially good ($R_k(S_2) = 2$ for all $k$), so we may assume $n > 2$.
The $k$-color Ramsey number of $S_n$ is $k(n-2) + \ep$, where $\ep = 1$ if $n$ is odd and $k$ is even, and $\ep = 2$ otherwise \cite{BuRo}.
So $\chi(G) \geq k(n-2)+1$.\footnote{Note that in the case $\ep = 2$, this result is $1$ stronger than it needs to be.}
By Theorem \ref{stars}, every $k$-coloring of the edges of $G$ contains a monochromatic $S_n$.
\end{proof}

\section{An application of Tur\'an numbers}
\label{sec:2}
Given a graph $H$ and an integer $n$, the Tur\'an number $\ex(n,H)$ is the maximum number of edges in a graph on $n$ vertices which does not have $H$ as a subgraph.
It is standard to estimate the Ramsey numbers $R(H_1, \dots, H_k)$ using $\ex(n,H_i)$ - here we follow this tradition.
\begin{theorem}
\label{exnum}
Let $G = (V, E)$ and $H_1, \dots, H_m$ be graphs, with $n = |V|$.
Suppose that $G$ is not a complete graph or an odd cycle, and suppose 
\[ \chi(G) \geq 1+\frac{2}{n}\sum_i\ex(n, H_i) \]
Then every $k$ coloring of the edges of $G$, there is some color $i$ which contains a monochromatic $H_i$.
\end{theorem}

\begin{proof}
Let $N = 1+\frac{2}{n}\sum_i\ex(n, H_i)$.
We may suppose that $G$ is a minimal $N$-chromatic graph.
Since $G$ is minimal, each vertex must have degree at least $N-1$.
And by Brooks' theorem, since $G$ is neither complete nor an odd cycle, $G$ must have a vertex of degree $N$.
So
\[ 2|E| = \sum_{v \in V} \deg(v) \geq n(N-1) + 1 = 1+ 2\sum_i\ex(n,H_i) \]
And thus $|E| > \sum_i\ex(n, H_i)$.
Therefore, in any coloring of the edges of $G$ by $k$ colors, there must be a color class $i$ with more than $\ex(n, H_i)$ edges.
Hence, color $i$ must have a monochromatic $H_i$.
\end{proof}

\begin{corollary}
$P_4$ is $k$-good for $k \not = 3$.
\end{corollary}

\begin{proof}
Let $k \not = 3$.
Let $G$ be a graph with chromatic number $R_k(P_4)$.
The $k$-color Ramsey number $R_k(P_4)$ is at least $2k+1$, for all $k \not = 3$ \cite{Ir, Wall}.\footnote{Again, when $R_k(P_4) > 2k+1$, this result is stronger than necessary.}
Further, we know that $\ex(n, P_4) \leq n$ \cite{ErGa}.
If $G$ is complete, then by definition any $k$-coloring of its edges has a monochromatic $P_4$.
If $G$ is an odd cycle, then $\chi(G) = 3$, so $k = 1$. 
Since $P_4$ is acyclic, it is $1$-good.
Otherwise, we have
\[ 1 + \frac{2}{n}\sum_{i=1}^k\ex(n, H) \leq 2k +1 \leq \chi(G) \]
So we may apply Theorem \ref{exnum}.
\end{proof}

Theorem \ref{exnum} is not sufficient to show that $P_4$ is $3$-good.
We have $\ex(n, P_4) \leq n$ and $R_3(P_4) = 6$ \cite{Ir}.
Then
\[ 2\ex(n, P_4) \frac{3}{n} + 1 \leq 7 \]
But to apply Theorem \ref{exnum}, we would need $2\ex(n, P_4) \frac{3}{n} +  1 \leq 6$.
Similarly, Theorem \ref{exnum} is also not sufficient to show that $P_5$ is $2$-good.
A more involved proof will show that $P_5$ is $2$-good, and since $R_2(P_5) = R(P_4, P_5) = 6$ \cite{GeGy}, it will immediately imply that $(P_4, P_5)$ is also good.
However, we can give a simple proof that $(P_4, P_5)$ is good from Theorem \ref{exnum}.

\begin{corollary}
$(P_4, P_5)$ is good.
\end{corollary}

\begin{proof}
We have $R(P_4, P_5) = 6$ \cite{GeGy}.
Let $G$ be a $6$-chromatic graph.
$G$ is not an odd cycle.
If $G$ is complete, then by $R(P_4, P_5) = 6$, every $2$-coloring of its edges has a $P_4$ in color $1$ or a $P_5$ in color $2$.
Otherwise, we can apply Theorem \ref{exnum}, since
\[ 1 + \frac{2}{n}(\ex(n, P_4) + \ex(n, P_5)) \leq 1 + \frac{2}{n}\left(n + \frac{3}{2}n\right)  = 6 = \chi(G) \qedhere \]
\end{proof}

The question of whether $P_4$ is $3$-good seems unique, as it evades proof by both this technique and the one presented in section \ref{sec:3}.
One possible method of attack is to notice that the only $P_4$-free graphs are (disjoint unions of) stars and triangles.
So if a $3$-colored graph is to have no monochromatic $P_4$, each of the three color classes must be the disjoint union of stars and triangles.

In the case where each color class is made up only of triangles, we have an interesting reduction.
Suppose we have a vertex-critical $6$-chromatic graph $G$ and a $3$-coloring of the edges such that each monochromatic connected component is a triangle.
Since $G$ has minimum degree at least $5$, every vertex must be in $3$ monochromatic triangles, so in fact $G$ must be $6$-regular.
If $P_4$ is $3$-good, no such $6$-chromatic $6$-regular graph with a decomposition into three sets of triangles can exist.

We can rephrase this as a problem about hypergraphs.
Given $G$, we make a hypergraph $\hH$ by replacing each monochromatic triangle with a hyperedge.
The dual $\hH^\star$ of $\hH$ is $3$-uniform, $3$-regular, $3$-partite, and linear (every pair of edges intersect in at most one vertex).
And importantly, $\chi(G) = \chi'(\hH^\star)$, where $\chi'$ is the chromatic index - the fewest number of colors needed to color the edges of the hypergraph such that no two intersecting edges have the same color.
Further, we can create every such dual hypergraph from some $P_4$-avoiding $G$.
Thus, to show that $P_4$ is good, it is necessary (but not sufficient) to answer the following question in the affirmative:

\begin{question}
Let $\hH$ be a $3$-uniform, $3$-regular, $3$-partite, linear hypergraph. Is 
\\* ${\chi'(\hH) \leq 5}$?
\end{question}

\section{A new method to establish that graphs are $2$-good}
\label{sec:3}
First, we fix some notation.
Consider a partially oriented graph $G$ with edges colored red and blue.
If $T$ is a subgraph of $G$ and $v$ is a vertex of $T$, then $d_T(v)$, $d_T^-(v)$, and $d_T^+(v)$ denote respectively the unoriented degree of $v$, the in-degree of $v$, and the out-degree of $v$ within $T$.

We denote the red subgraph of $G$ by $G_r$ and the blue subgraph by $G_b$.
We call the connected components of $G_r$ and $G_b$ the {\em monochromatic parts} of $G$.
If $T$ is a monochromatic part of $G$ and $T$ has no oriented edges, $T$ is a {\em main part} of $G$.
If $T$ does have oriented edges, we define two subsets of the vertices of $T$.
$T_-$ is the set of vertices $v$ of $T$ such that $d_T^-(v) \geq 2$ and there are no nontrivial directed paths from $v$ to $u$ such that $d_T^-(u) \geq 2$, for any vertices $u$ in $T$. 
And $T_+$ is the set of vertices $v$ such that there is a nontrivial directed path from $v$ to $u$, for some $u \in T_-$.

\begin{definition}
\label{boundeddef}
A partial orientation of a graph $T$ is {\em $(s,t)$-bounded} if 
\begin{enumerate}
\item[(1)] For every vertex $v$ with $d_T^-(v) > 0$,
\[  d_T(v)+ d_T^-(v) + \min\{1,d_T^+(v)\}   \le s \]
\item[(2)] For every vertex $v$ of $T$, 
\[ d_T(v) + \min\{1,d_T^+(v)+d_T^-(v)\}  \le t-1 \]
\item[(3)] if $T_-$ (equivalently, $T_+$) is nonempty, then $|T_-| > |T_+|$
\end{enumerate}
A partial orientation is {\em $(n,s,t)$-bounded} if it is $(s,t)$-bounded and
\begin{enumerate}
\item[(4)] If $T$ has more than $n$ vertices, then at least one edge of $T$ is oriented.
\end{enumerate}
A family of graphs $\tT$ is {\em $n$-bounded} if there are nonnegative integers $s$ and $t$ such that $n>s+t-1$, $n>2s+2$, and every graph in $\tT$ has an $(n,s,t)$-bounded partial orientation.
\end{definition}

\begin{lemma}
\label{5bounded}
The family of $P_5$-free graphs is $5$-bounded.
\end{lemma}

\begin{proof}
A {\em pendant edge} is an edge incident with a degree $1$ vertex.
For each monochromatic part of $G$, orient each pendant edge towards the associated leaf.

\begin{myfigure}
\begin{graph}
\draw (0,0) node(v1) {};
\draw (1,1) node(v2) {};
\draw (2,0) node(v3) {};
\draw (1,-1) node(v4) {};
\path[draw,blue] (v1) -- (v2);
\path[draw,blue] (v1) -- (v3);
\path[draw,blue] (v1) -- (v4);
\path[draw,blue] (v2) -- (v3);
\path[draw,blue] (v2) -- (v4);
\path[draw,blue] (v3) -- (v4);

\draw (3.22, -0.71) node(v5) {};
\draw (3.22, 0.71) node(v6) {};
\draw (4.59, -0.34) node(v7) {};
\draw (4.20, 0.1) node(v8) {};
\path[draw,red,dashed] (v3) -- (v5);
\path[draw,red,dashed] (v3) -- (v6);
\path[draw,red,dashed] (v5) -- (v6);
\path[redarrow,dashed] (v5) -- (v1);
\path[redarrow,dashed] (v5) -- (v4);
\path[redarrow,dashed] (v5) -- (v7);
\path[redarrow,dashed] (v5) -- (v8);

\draw (4.10, 1.0) node(v9) {};
\draw (4.80, 0.8) node(v10) {};
\draw[draw,blue] (v9) -- (v10);
\draw[bluearrow] (v9) -- (v5);
\draw[bluearrow] (v9) -- (v6);
\draw[bluearrow] (v10) -- (v7);
\draw[bluearrow] (v10) -- (v8);
\end{graph}
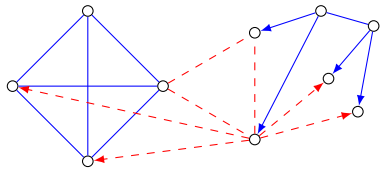
\caption{The partial orientation of a $2$-colored graph with no monochromatic $P_5$.}
\end{myfigure}

We call this the {\em standard orientation} of the pendant edge.
It is easy to see that this is a $(5,1,4)$-bounded orientation:
\begin{enumerate}
\item[(1)] There are no vertices with an incoming edge and an outgoing edge of the same color. 

\item[(3)] No vertex has in-degree greater than $1$, so $T_-=T_+=\varnothing$.

\item[(4)] Each component without pendant edges has at most $4$ vertices.  

\item[(2)] In a component with pendant edges, if the degree one vertices are removed, the resulting graph must be a $1$, $2$, or $3$-clique.
Along with observation (4), this implies that every vertex has at most $3$ unoriented edges per color.
\qedhere
\end{enumerate}
\end{proof}

We need the following technical lemma.
\begin{lemma}
\label{technicallemma}
Let $G=(V,E)$ be a $2$-colored graph with a given partial orientation.
Let $n$, $s$, and $t$ be nonnegative integers such that $n>s+t-1$ and $n>2s+2$.
Suppose that $G$ has minimum degree at least $n$.
If the partial orientation of each monochromatic part is $(s,t)$-bounded, then every monochromatic part of $G$ is a main part.
\end{lemma}
The proof of Lemma \ref{technicallemma} is presented after the proof of Theorem \ref{maintheorem}, once the application is clearly in mind.

\begin{theorem}
\label{maintheorem}
Let $H$ be an arbitrary graph.
If the family of $H$-free graphs is $\left(R_2(H)-1\right)$-bounded, then $H$ is $2$-good.
\end{theorem}

\begin{proof}
Let $G$ be a minimal $R_2(H)$-chromatic graph, and consider a $2$-coloring of $G$.
Assume for contradiction that $G$ has no monochromatic subgraph $H$.
Consider an orientation of $G$ for which the orientation of the monochromatic parts witnesses the $(R_2(H)-1)$-boundedness of $G$.
$G$ has minimum degree $R_2(H)-1$, so by Lemma \ref{technicallemma}, every monochromatic part of $G$ is a main part with this orientation.

Define a multigraph $G'$ as follows.
Let the vertex set of $G'$ be the set of main parts of $G$.
For any two main parts, there is an edge between them for each vertex they have in common.
Since every main part has at most $R_2(H)-1$ vertices and every vertex of $G$ has degree at least $R_2(H)-1$, each vertex must be in two main parts, so every vertex in $G$ is represented as an edge in $G'$.
Note that $G'$ is bipartite, since no two main parts of the same color share a vertex.
Further, $G'$ has maximum degree $R_2(H)-1$ (the maximum size of a main part).
By K\"onig's line coloring theorem, $G'$ has a proper edge coloring with $R_2(H)-1$ colors.
Since each vertex in $G$ is identified with an edge in $G'$, this edge coloring induces a vertex coloring of $G$.
If two vertices are adjacent in $G$, they must be in the same main part, so the corresponding edges in $G'$ are adjacent.
Therefore, the induced vertex coloring is proper, so $\chi(G) \leq R_2(H)-1$; contradiction.
\end{proof}

\begin{corollary}
$P_5$ is $2$-good.
\end{corollary}
\begin{proof}
The $2$-color Ramsey number of $P_5$ is $R_2(P_5) = 6$ \cite{GeGy}.
By Lemma \ref{5bounded}, the family of $P_5$-free graphs is $5$-bounded, so by Theorem \ref{maintheorem}, $P_5$ is $2$-good.
\end{proof}

\begin{proof}[Proof of Lemma \ref{technicallemma}]
Let $G=(V,E)$, $n$, $s$, and $t$ be given as in Lemma \ref{technicallemma}.
Elementarily, $\sum_{v \in V} \indeg(v) = \sum_{v \in V} \outdeg(v)$.
Let $\tT$ be the set of all monochromatic parts of $G$.
Then define
\[ G_- := \bigcup_{T \in \tT}T_- \qquad \text{ and } \qquad G_+ := \bigcup_{T \in \tT}T_+ \]
Also define 
\[ X := \{v \in V\setminus(G_- \cup G_+) \pipe \indeg(v) > 0 \text{ or } \outdeg(v) > 0 \} \]
Note that every vertex incident with an oriented edge is in $G_-$, $G_+$, or $X$.
Then we have
\[ 0 = \sum_{v \in V} \indeg(v) - \outdeg(v) = \sum_{v \in G_-\cup G_+} \indeg(v) - \outdeg(v) + \sum_{v \in X} \indeg(v) - \outdeg(v) \]
To conclude that $G_-$, $G_+$, and $X$ must be empty, we show in Claim \ref{X empty} that the sum over $X$ is positive if $X$ is nonempty and then in Claim \ref{G-+ empty} that the sum over $G_- \cup G_+$ is positive (if $G_- \cup G_+$ is nonempty).
\begin{claim}
\label{X empty}
Let $v$ be a vertex in $G$ which is incident with at least one oriented edge.
If $v$ has at most one incoming edge of each color (i.e. $v \not \in T_-\cup T_+$ for any $T$), then $\outdeg(v) > \indeg(v)$.
If $v$ is in $T_-$ for just one $T$, then $d_G^+(v) \geq 2$.
And if $v$ is in $T_{1-}$ and $T_{2-}$ for two distinct $T_1,T_2\in \tT$, then $d_G^+(v) \geq 4$.

This will immediately imply that $\sum\limits_{v \in X} \indeg(v) - \outdeg(v)$ is positive if $X$ is nonempty.
\end{claim}
\begin{proof}
Consider the possible combinations of colors for the incoming edges of $v$.
\medskip

\case{0} If $v$ has no incoming edges, then the oriented edge incident with $v$ must be an outgoing edge.
\medskip

\case{1} If $v$ has incoming edges of only one color, we may assume without loss of generality that they are red.
By (1) and (2) (in Definition \ref{boundeddef})
\[ \big(d_\red(v) + d_\red^-(v)\big) + \big(d_\blue(v) + d_\blue^-(v)\big) \leq s + (t-1)\] 
Since $n>s+t-1$, $v$ must have an outgoing edge.
If the outgoing edge is red, then use (1), and if the outgoing edge is blue, use (2).
In either case, we can strengthen the above to 
\[ \big(d_\red(v) + d_\red^-(v)\big) + \big(d_\blue(v) + d_\blue^-(v)\big) \leq s + (t-1) - 1\] 
so $v$ must have a second outgoing edge.
If $v$ has in-degree at most one in each color, then $v$ has only one incoming edge, so $d_G^+(v) > d_G^-(v)$.
\medskip

\case{2} Suppose $v$ has incoming edges of $2$ colors.
Then $d_\red(v)+d_\red^-(v) \leq s$ and $d_\blue(v)+d_\blue^-(v) \leq s$ by (1).
Since $n > 2s+2$, $v$ has at least three outgoing edges.
So there is some color - say, red - such that $v$ has a red incoming edge and a red outgoing edge.
Then by (1), 
\[ \big(d_\red(v) + d_\red^-(v)\big) + \big(d_\blue(v) + d_\blue^-(v)\big) \leq (s-1) + s = 2s-1 \] 
so $v$ has at least $4$ outgoing edges.
If $v$ has in-degree at most one in each color, then $v$ has only two incomings edges, so $d^+(v) > d^-(v)$.
\end{proof}

\begin{claim}
\label{G-+ empty}
$\sum\limits_{v \in G_-\cup G_+} \indeg(v) - \outdeg(v)$ is postive if $G_- \cup G_+$ is nonempty.
\end{claim}
\begin{proof}
Assume $G_- \cup G_+$ is nonempty.
Since $T_- =\varnothing$ iff $T_+=\varnothing$ for all $T \in \tT$, this implies that $G_-$ and $G_+$ are both nonempty.
Let $e=(u,v)$ be an oriented (from $u$ to $v$) edge in $G$ with $v \in G_- \cup G_+$ and $u \not \in G_- \cup G_+$.
Suppose $e$ is colored red (blue).
Since $v \in G_- \cup G_+$, there must be at least one $T$ such that $v \in T_-$ or $v \in T_+$.
But observe that $v$ cannot be in $T_-$ or $T_+$ for any red (blue) $T \in \tT$: otherwise $u$ would be in $T_+$, and thus in $G_+$.
In particular, $v$ cannot be in both a $T_{1-}$ and a $T_{2+}$ for any $T_1,T_2 \in \tT$, since $T_1$ and $T_2$ would have to be of different colors. 
So $v \not \in G_- \cap G_+$.
With this motivation, define $G'_- := G_- \setminus G_+$ and $G'_+ := G_+ \setminus G_-$.
The above argument shows that
\begin{equation}
\label{star}
\indeg(G_- \cup G_+) \leq |G'_- \cup G'_+| = |G'_-| + |G'_+|
\tag{$\star$}
\end{equation}
Since no vertex $v \in G_+$ is in $T_{1-}$ and $T_{2-}$ for distinct $T_1,T_2 \in \tT$,
\[ |G'_+| = |G_+|-|G_+\cap G_-| = |G_+| - \sum_{T \in \tT}|G_+ \cap T_-| \leq \sum_{T \in \tT} |T_+| - \sum_{T \in \tT}|G_+ \cap T_-| \]
By criterion (3), $|T_+| < |T_-|$ if $T_-$ and $T_+$ are nonempty, for all $T \in \tT$ (and by assumption, some there is at least one such $T$), so
\[ |G'_+| <  \sum_{T \in \tT} |T_-| - \sum_{T \in \tT}|G_+ \cap T_-| \]
For any $T \in \tT$, the vertices in $T_-$ but not in $G_+$ are exactly those in $T_-$ and $G'_-$.
Hence
\[ |G'_+| < \sum_{T \in \tT}|T_-\cap G'_-| \]
It is easy to see that this bound works for $|G'_-|$ too: $|G'_-| \leq \sum_{T \in \tT}|T_-\cap G'_-|$.
Putting these together with (\ref{star}), we have
\[ \indeg(G_- \cup G_+) < 2\sum_{T \in \tT}|T_- \cap G'_-| \]

Every vertex in $G'_-$ has at least two outgoing edges, by Claim \ref{X empty}.
Those which are double-counted by $\sum_{T \in \tT}|T_- \cap G'_-|$ are exactly those which are in $T_{1-}$ and $T_{2-}$ for distinct $T_1,T_2 \in \tT$.
By Claim \ref{X empty}, these have at least four outgoing edges.
If all of these outgoing edges went out of $G_+ \cup G_-$, we would be done; however, some of the edges may go to vertices in $G_+\cup G_-$.
Let $k$ be the number of these edges from $G'_-$ to $G_+ \cup G_-$.
So we have $2\left(\sum_{T \in \tT}|T_- \cap G'_-|\right) -k \leq \outdeg(G_- \cup G_+)$.
Let $(u,v)$ be one of these edges, with $u \in G'_-$ and $v \in G_-\cup G_+$.
We show that $v$ can have no incoming edges $(w,v)$ with $w \not \in G_- \cup G_+$.
Suppose $(u,v)$ is red.
There is some $T$ such that $v$ is in $T_-$ or $T_+$.
If $(w,v)$ is blue, then either $u$ or $w$ is in $T_+$ (depending on the color of $T$).
If $(w,v)$ is red, then $v$ has two incoming edges, so there is some red $T'$ (maybe $T$) such that $v \in T'_-$.
But then $u,w \in T'_+$, a contradiction.
Similarly, if $(u,v)$ is blue, $v$ can have no other incoming edges.
Therefore
\[ \indeg(G_- \cup G_+) < 2\left(\sum_{T \in \tT}|T_- \cap G'_-|\right) -k \leq \outdeg(G_- \cup G_+) \qedhere\]
\end{proof}
Therefore $G_-$, $G_+$, and $X$ are all empty, so $G$ is made up entirely of main parts.
\end{proof}

\subsection{$P_6$ and $P_7$ are $2$-good}
\begin{definition}
A {\em pendant star} in a graph $G$ is a star $S$ where exactly one vertex in the star has an edge which is not part of the star, and that vertex is one of leaves of $S$.
A {\em pendant triangle} in a graph $G$ is a triangle where exactly one vertex in the triangle has an edge which is not part of the triangle.

The {\em standard orientation} of a pendant star or triangle is as follows.
Let $v$ be the vertex of the pendant subgraph which is connected to the rest of the graph.
\begin{itemize}
\item Orient all edges of a pendant star away from $v$, and
\item in a pendant triangle, orient the two edges adjacent to $v$ away from $v$, and leave the third unoriented.
\end{itemize}
\end{definition}

\begin{theorem}
\label{P6}
$P_6$ is $2$-good.
\end{theorem}

\begin{proof}
The $2$-color Ramsey number of $P_6$ is $R_2(P_6) = 8$ \cite{GeGy}.
To show that the family of $P_6$-free graphs is $7$-bounded, we construct a $(7, 2, 5)$-bounded partial orientation for each $P_6$-free graph.

In each monochromatic part, we will partially orient the edges by (usually) orienting away from a longest cycle.
Let $T$ be a monochromatic part of $G$.
Orient it as follows:

\begin{case}{1} Suppose $T$ has a $5$-cycle.
Then $T$ can have no vertices not in this cycle.
In this case, we orient none of the edges - $T$ is a main part.
To see that this is indeed a $(7,2,5)$-bounded orientation of $T$, note that (2) and (4) are true since $d(v) \leq 4$ for $v\in T$ and $|T|<7$.

\twographs{
\begin{graph}
\fivecycle{c}
\draw (c 1) -- (c 3);
\draw (c 2) -- (c 4);
\end{graph}
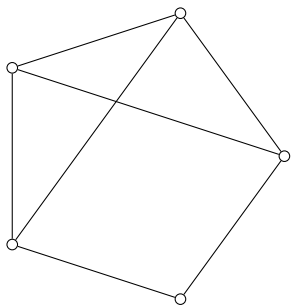
}{If $T$ has a $5$-cycle, it is a main part.}{
\begin{graph}
\def \n {2cm}
\node [label={[label distance=1mm]above:$a$}] (a) at (0, 0) {};
\node [label={[label distance=1mm]above:$b$}] (b) at (\n,\n) {};
\node [label={[label distance=1mm]above:$c$}] (c) at (2*\n, 0) {};
\node [label={[label distance=1mm]below:$d$}] (d) at (\n,-\n) {};
\draw[arrow] (a) -- (b);
\draw[arrow] (a) -- (d);
\draw[arrow] (c) -- (b);
\draw[arrow] (c) -- (d);
\draw (a) -- (c);

\draw (\n,1.5) node(v2) {};
\draw (\n,1.0) node(v3) {};
\draw (\n,0.5) node(v4) {};
\draw (\n,-1.5) node(v5) {};
\foreach \s in {2,...,5}
{
    \draw[arrow] (a) -- (v\s); \draw[arrow] (c) -- (v\s);
}

\foreach \s in {1,...,5}
{
    \node[draw] (a \s) at ({360/10 * (\s + 3)}:\n/2) {};
    \draw[arrow] (a) -- (a \s); 
}

\draw (5*\n/2,0) node(c 1) {};
\draw (5*\n/2,\n/2) node(c 2) {};
\draw (5*\n/2,-\n/2) node(c 3) {};
\foreach \s in {1,...,3}
{
    \draw[arrow] (c) -- (c \s); 
}
\end{graph}
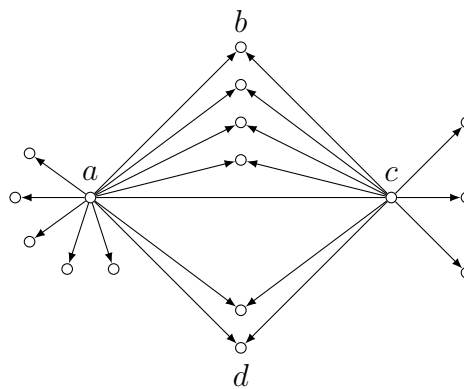
}{If there are at least three $2$-edge $a-c$ paths, then ${T_+ = \{a,c\}}$ and $|T_-| \geq 3$.}
\end{case}

\begin{case}{2} Suppose a longest cycle $C$ in $T$ is a $4$-cycle.
Let $a$, $b$, $c$, $d$ be the vertices of $C$, in that order.
Consider one of the vertices of $C$; without loss of generality, assume it's $a$.
Then there can be no path from $a$ to $b$ or from $a$ to $d$ without some edge in $C$ (if there is, then there is a larger cycle than $C$).
Further, if there is a path from $a$ to $c$ without using an edge in $C$, it must have $1$ or $2$ edges, for the same reason.
A one-edge path from $a$ to $c$ is just the edge $\{a, c\}$, which there can only be one of.
If there is a two-edge $a-c$ path without using edges in $C$, note that there can be no two-edge $b-d$ path without using edges in $C$.
Now we consider two cases:
\begin{itemize}
\item Suppose there is a two-edge $a-c$ or $b-d$ path which avoids edges in $C$; without loss of generality, assume it is an $a-c$ path.
There may be many such paths - let $b = v_1, v_2, \dots, v_m = d$ be the vertices such that $\{a, v_i\}$ and $\{v_i, c\}$ are edges, for $i = 1, \dots, m$.

Note that $m \geq 3$.
None of the vertices $v_1, \dots, v_m$ can have any other edges: if there is some $\{v_i, v_j\}$ edge, then it makes a $5$-cycle, and if one of $v_1, \dots, v_m$ has any other edge, it makes a $P_6$.
Note also that all of the edges from $a$ or $c$ which are not of the form $\{a, v_i\}$ or $\{c, v_i\}$ are pendant edges.
So all of the edges are adjacent to either $a$ or $c$; orient those adjacent to $a$ away from $a$, and orient those adjacent to $c$ away from $c$.
The one possible exception is the edge $\{a, c\}$: leave this one unoriented.
Since $m \geq 3$, this satisfies criterion (3): $T_+ = \{a,c\}$ and $|T_-| =m$.

\item Suppose that there is no two-edge $a-c$ or $b-d$ path avoiding $C$.
Note that each edge which has an endpoint which is not one of $a$, $b$, $c$, or $d$ is a pendant edge.
Orient the pendant edges with the standard orientation.

\begin{myfigure}
\begin{graph}
\def \n {1.5cm}
\node [label={[label distance=1mm]above:$a$}] (a) at (0, 0) {};
\node [label={[label distance=1mm]above:$b$}] (b) at (\n,\n) {};
\node [label={[label distance=1mm]above:$c$}] (c) at (2*\n, 0) {};
\node [label={[label distance=1mm]below:$d$}] (d) at (\n,-\n) {};
\draw (a) -- (b);
\draw (a) -- (d);
\draw (c) -- (b);
\draw (c) -- (d);
\draw (a) -- (c);

\foreach \s in {1,...,5}
{
    \node[draw] (a \s) at ({360/10 * (\s + 3)}:3*\n/4) {};
    \draw[arrow] (a) -- (a \s); 

}

\draw (23*\n/8,0) node(c 1) {};
\draw (11*\n/4,\n/2) node(c 2) {};
\draw (11*\n/4,-\n/2) node(c 3) {};
\foreach \s in {1,...,3}
{
    \draw[arrow] (c) -- (c \s); 
}
\end{graph}
\hspace{1cm}
\begin{graph}
\def \n {1.5cm}
\node [label={[label distance=1mm]above:$a$}] (a) at (0, 0) {};
\node [label={[label distance=1mm]above:$b$}] (b) at (\n,\n) {};
\node [label={[label distance=1mm]right:$c$}] (c) at (2*\n, 0) {};
\node [label={[label distance=1mm]below:$d$}] (d) at (\n,-\n) {};
\draw (a) -- (b);
\draw (a) -- (d);
\draw (c) -- (b);
\draw (c) -- (d);
\draw (a) -- (c);
\draw (b) -- (d);

\foreach \s in {1,...,5}
{
    \node[draw] (a \s) at ({360/10 * (\s + 3)}:3*\n/4) {};
    \draw[arrow] (a) -- (a \s); 

}

\end{graph}
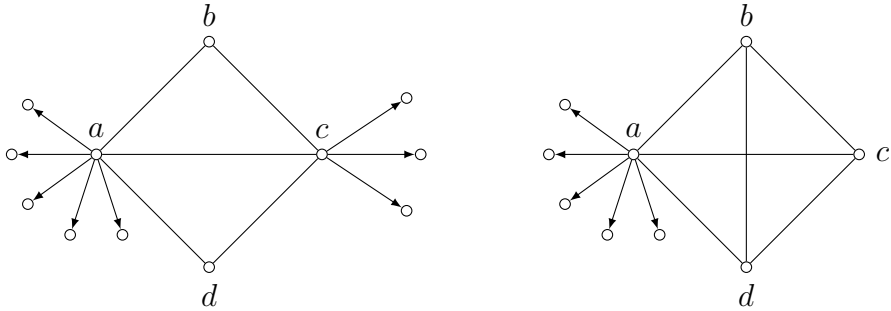
\caption{If there are only two $2$-edge $a-c$ paths, just orient the pendant edges.}
\end{myfigure}
\end{itemize}

In both cases, condition (2) holds, and the others hold trivially.
So we have a $(7,2,5)$-bounded orientation of $T$.
\end{case}

\begin{case}{3} Suppose a longest cycle $C$ in $T$ is a $3$-cycle.
Then there can be no path between two vertices of $C$ that uses some edges not in $C$, since such a path would imply the existence of a larger cycle.
For any vertex $v$ of $C$, any path from $v$ which does not use edges in $C$ can have at most two edges.
So each vertex of $C$ can have only pendant edges, pendant stars, or pendant triangles.
Orient these according to the standard orientation.

\begin{myfigure}
\begin{graph}
\def \n {1.5cm}
\draw (0,0) node(a) {};
\foreach \s in {1,...,8}
{
    \node[draw] (a \s) at ({360/14 * (\s-1)}:7*\n/8) {};
    \draw[arrow] (a) -- (a \s); 
}
\draw (a 1) -- (a 2);
\draw (a 3) -- (a 4);
\draw (a 5) -- (a 6);
\draw (a 7) -- (a 8);

\foreach \s in {7,...,9}
{
    \node[draw] (b \s) at ({360/12 * (\s + 1)}:3*\n/4) {};
    \draw[arrow] (a) -- (b \s); 
}
\foreach \s in {-1,...,1}
{
    \node[draw] (v \s) at (\s*\n/2, -3*\n/2) {};
    \draw[arrow] (b 8) -- (v \s); 
}
\end{graph}
\hspace{1.5cm}
\begin{graph}
\def \n {1.5cm}
\draw (0,0) node(a) {};
\draw (1.22*\n*2/3, -0.71*\n*2/3) node(b) {};
\draw (1.22*\n*2/3, 0.71*\n*2/3) node(c) {};
\draw (a) -- (b);
\draw (c) -- (b);
\draw (a) -- (c);

\foreach \s in {1,...,4}
{
    \node[draw] (a \s) at ({360/11 * (\s + 3)}:3*\n/4) {};
    \draw[arrow] (a) -- (a \s); 
}

\node[draw] (b 1) at (1.2*\n, -0.9*\n) {};
\node[draw] (b 2) at (0.9*\n, -1*\n) {};
\node[draw] (b 3) at (0.6*\n, -0.9*\n) {};

\node[draw] (c 1) at (1.2*\n, 0.9*\n) {};
\node[draw] (c 2) at (0.9*\n, 1*\n) {};
\node[draw] (c 3) at (0.6*\n, 0.9*\n) {};

\foreach \s in {1,...,3}
{
    \draw[arrow] (b) -- (b \s); 
    \draw[arrow] (c) -- (c \s); 
}
\end{graph}
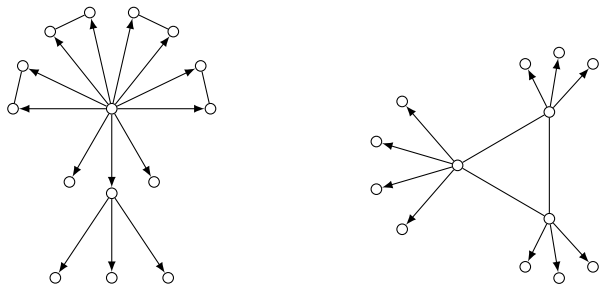
\caption{If the longest cycle in $T$ is a $3$-cycle, orient pendant stars, triangles, and edges outward.}
\label{P_6 3cycle}
\end{myfigure}
\end{case}

\begin{myfigure}
\begin{graph}
\def \n {1cm}
\node [label={[label distance=1mm]left:$r$}] (s) at (2/3*\n,2*\n) {};
\draw (2/3*\n,\n) node(t 1) {};
\draw[arrow] (s) -- (t 1);

\draw (0,0) node(a) {};
\draw (4/3*\n,0) node(t 3) {};
\draw[arrow] (t 1) -- (a);
\draw[arrow] (t 1) -- (t 3);

\foreach \s in {1,...,4}
{
    \node[draw] (a \s) at ({360/11 * (\s + 5)}:\n) {};
    \draw[arrow] (a) -- (a \s); 
}

\foreach \s in {-1,...,1}
{
    \node[draw] (v \s) at (\s*\n/2-0.2*\n, -2*\n) {};
    \draw[arrow] (a 3) -- (v \s); 
}

\end{graph}
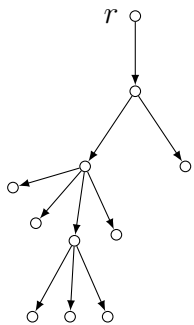
\caption{Orient trees away from some arbitrary root.} 
\label{P_6 tree}
\end{myfigure}

\begin{case}{4} Suppose $T$ has no cycles.
Then choose a root vertex $r$ arbitrarily, and orient every edge away from $r$.
\end{case}

In cases 3 and 4, one can easily check properties (1) and (2) ((3) and (4) are trivial). 
Note that (1) holds with equality at many vertices on Figures \ref{P_6 3cycle} and \ref{P_6 tree}.
\end{proof}

\begin{theorem}
\label{P7}
$P_7$ is $2$-good.
\end{theorem}

\begin{proof}
The $2$-color Ramsey number of $P_7$ is $R_2(P_7) = 9$ \cite{GeGy}.
With a natural extension of the partial orientations from the previous proof, we can give each $P_7$-free graph an $(8,2,6)$-bounded orientation.
Let $T$ be a monochromatic part of $G$, and orient it as follows:

\begin{case}{1} If $T$ has a $6$-cycle, then every vertex of $T$ is in that cycle.
Leave all the edges unoriented; $T$ is a main part.

\twographs{
\begin{graph}
\sixcycle{c}
\draw (c 1) -- (c 3);
\draw (c 2) -- (c 4);
\draw (c 3) -- (c 5);
\draw (c 2) -- (c 5);
\end{graph}
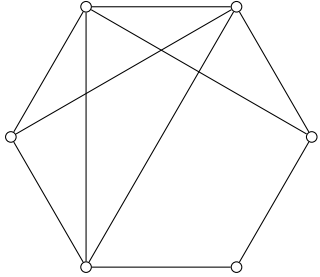
}{If $T$ has a $6$-cycle, it is a main part.}{
\begin{graph}
\def \n {1.5cm}
\fivecycle{a}

\draw (a 1) -- (a 3);
\draw (a 2) -- (a 4);
\draw (a 4) -- (a 1);

\draw (\n/2,\n/5) node(c 1) {};
\draw (3*\n/5,0) node(c 2) {};
\draw (\n/2,-\n/5) node(c 3) {};
\foreach \s in {1,...,3}
{
    \draw[arrow] (a 1) -- (c \s); 
}

\draw (-4*\n/7,-\n/5) node(b 1) {};
\draw (-\n/2,-\n/5*2) node(b 2) {};
\draw (-\n/3,-\n/2) node(b 3) {};
\foreach \s in {1,...,3}
{
    \draw[arrow] (a 4) -- (b \s); 
}
\end{graph}
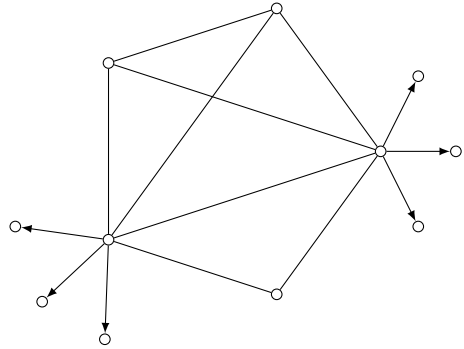
}{When $T$ has no extra two-edge paths, orient only the pendant edges.}
\end{case}

\begin{case}{2} Suppose a longest cycle $C$ in $T$ is a $5$-cycle, with vertices $a$, $b$, $c$, $d$, $e$, in that order.
Let $v\not=a,b,c,d,e$ be a vertex of $T$ such that $d_T(v) > 1$.
Without loss of generality, we may assume that for all such vertices, $\{a,v\}$ and $\{v,c\}$ are edges in $T$.
Let $m$ be the number of vertices (including $b$, but not $d$ or $e$) such that $\{a,v\}$ and $\{v,c\}$ are edges in $T$.
And call these vertices $b=v_1, \dots, v_m$.
It is easy to check that there can be no edge $\{v_i,v_j\}$ for any distinct $i,j$.
\begin{itemize}
\item 
If $m = 1$, then every edge with an endpoint not in the cycle $C$ is a pendant edge.
Orient the pendant edges away from the cycle, and leave the rest unoriented.

\item $m=2$ is a special case.
We would like to orient towards $v_1$ and $v_2$, but criterion (3) would not be satisfied, since we would have $|T_-|=|T_+|=2$.
And we can't just orient the pendant edges, since then $a$ or $c$ could have unoriented degree $5$ and positive out-degree, violating criterion (2).
Instead, orient just $\{a,v_1\}$ and $\{v_2,c\}$ away from $a$ and $c$.
Orient the pendant edges with the standard orientation.
Leave the any other edges between two vertices of the cycle unoriented (if they exist).

\twographs{
\begin{graph}
\def \n {2cm}
\node [label={[label distance=1mm]left:$a$}] (a) at (0, 0) {};
\node [label={[label distance=-3mm]above:$b=v_1$}] (b) at (\n,\n) {};
\node [label={[label distance=1mm]right:$c$}] (c) at (2*\n,0) {};
\node [label={[label distance=1mm]36*8:$d$}] (d) at (3/2*\n,-\n) {};
\node [label={[label distance=1mm]36*7:$e$}] (e) at (1/2*\n,-\n) {};
\draw[very thick,arrow] (a) -- (b);
\draw (c) -- (b);

\draw (a) -- (c);
\draw (a) -- (d);
\draw (a) -- (e);
\draw (c) -- (d);
\draw (c) -- (e);
\draw (d) -- (e);

\node [label={[label distance=1mm]below:$v_2$}] (v2) at (\n,0.5*\n) {};
\draw (a) -- (v2); 
\draw[very thick, arrow] (c) -- (v2);

\foreach \s in {1,2,4,5}
{
    \node (a \s) at ({360/10 * (\s + 2)}:\n/2) {};
    \draw[arrow] (a) -- (a \s); 
}

\draw (2.4*\n,0.35*\n) node(c 1) {};
\draw (2.15*\n,0.47*\n) node(c 2) {};
\draw (2.25*\n,-0.44*\n) node(c 3) {};
\foreach \s in {1,...,3}
{
    \draw[arrow] (c) -- (c \s); 
}
\end{graph}
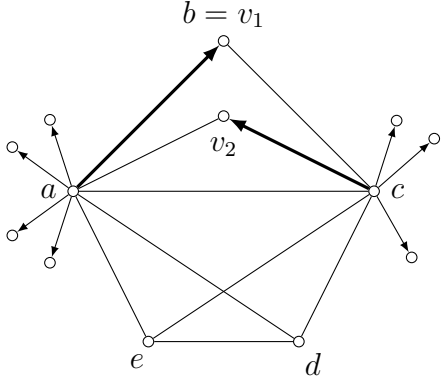
}{When $m=2$, orient just one edge towards each of $v_1$ and $v_2$.}{
\begin{graph}
\def \n {2cm}
\node [label={[label distance=2mm]left:$a$}] (a) at (0, 0) {};
\node [label={[label distance=-3mm]above:$b=v_1$}] (b) at (\n,\n) {};
\node [label={[label distance=1mm]right:$c$}] (c) at (2*\n,0) {};
\node [label={[label distance=1mm]36*8:$d$}] (d) at (3/2*\n,-\n) {};
\node [label={[label distance=1mm]36*7:$e$}] (e) at (1/2*\n,-\n) {};
\draw[arrow] (a) -- (b);
\draw[arrow] (c) -- (b);

\draw (a) -- (c);
\draw (a) -- (d);
\draw (a) -- (e);
\draw (c) -- (d);
\draw (c) -- (e);
\draw (d) -- (e);

\draw (\n,0.75*\n) node(v2) {};
\draw (\n,0.5*\n) node(v3) {};
\draw (\n,0.25*\n) node(v4) {};
\node [label={[label distance=0]below:$v_4$}] (v4) at (\n,0.25*\n) {};
\foreach \s in {2,...,4}
{
    \draw[arrow] (a) -- (v\s); \draw[arrow] (c) -- (v\s);
}

\foreach \s in {1,2,4,5}
{
    \node (a \s) at ({360/10 * (\s + 2)}:\n/2) {};
    \draw[arrow] (a) -- (a \s); 
}

\draw (2.4*\n,0.35*\n) node(c 1) {};
\draw (2.15*\n,0.47*\n) node(c 2) {};
\draw (2.25*\n,-0.44*\n) node(c 3) {};
\foreach \s in {1,...,3}
{
    \draw[arrow] (c) -- (c \s); 
}

\end{graph}
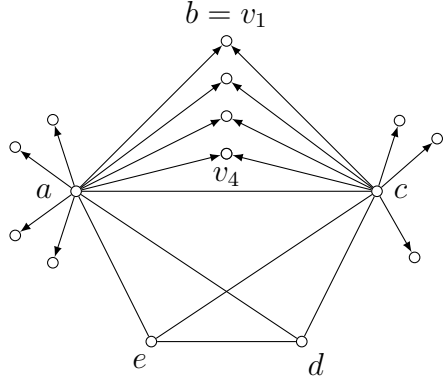
}{When $T$ has many $a-c$ paths, $T_-$ is $\{a,c\}$ and $|T_+| \geq 3$.}

\item 
Otherwise, $m \geq 3$.
Orient the edges $\{a, v_i\}$ away from $a$, for all $i=1, \dots, m$, orient the edges $\{c, v_i\}$ away from $c$ for all $i$.
Give pendant edges the standard orientation and leave other edges between vertices of the cycle unoriented.
We have $T_- = \{a,c\}$ and $T_+ = \{v_1, \dots, v_m\}$, so criterion (3) is satisfied.
\end{itemize}
\end{case}

\begin{myfigure}
\begin{graph}
\def \n {2cm}
\node [label={[label distance=1mm]above:$a$}] (a) at (0, 0) {};
\node [label={[label distance=1mm]above:$b$}] (b) at (\n,\n) {};
\node [label={[label distance=1mm]above:$c$}] (c) at (2*\n, 0) {};
\node [label={[label distance=1mm]below:$d$}] (d) at (\n,-\n) {};
\draw[arrow] (a) -- (b);
\draw[arrow] (a) -- (d);
\draw[arrow] (c) -- (b);
\draw[arrow] (c) -- (d);
\draw (a) -- (c);

\draw (\n,1.5) node(v2) {};
\draw (\n,1.0) node(v3) {};
\draw (\n,0.5) node(v4) {};
\draw (\n,-1.5) node(v5) {};
\foreach \s in {2,...,5}
{
    \draw[arrow] (a) -- (v\s); \draw[arrow] (c) -- (v\s);
}

\foreach \s in {1,...,5}
{
    \node (a \s) at ({360/10 * (\s + 3)}:\n/2) {};
    \draw[arrow] (a) -- (a \s); 
}
\draw (a 2) -- (a 3);
\draw (a 4) -- (a 5);

\node (a 6) at (-0.6*\n, 0.7*\n) {};
\node (a 7) at (-0.4*\n, 0.77*\n) {};
\node (a 8) at (-0.2*\n, 0.7*\n) {};
\node (a 9) at (-0.1*\n, 0.55*\n) {};
\foreach \s in {6,...,9}
{
    \draw[arrow] (a 1) -- (a \s); 
}

\draw (5*\n/2,0) node(c 1) {};
\draw (2.4*\n,0.4*\n) node(c 2) {};
\draw (2.3*\n,-\n/2) node(c 3) {};
\foreach \s in {1,...,3}
{
    \draw[arrow] (c) -- (c \s); 
}
\end{graph}
\hspace{1cm}
\begin{graph}
\def \n {2cm}
\node [label={[label distance=1mm]above:$a$}] (a) at (0, 0) {};
\node [label={[label distance=1mm]above:$b$}] (b) at (\n,\n) {};
\node [label={[label distance=1mm]above:$c$}] (c) at (2*\n, 0) {};
\node [label={[label distance=1mm]below:$d$}] (d) at (\n,-\n) {};
\draw[arrow] (a) -- (b);
\draw[arrow] (a) -- (d);
\draw[arrow] (c) -- (b);
\draw[arrow] (c) -- (d);
\draw (a) -- (c);

\draw (\n,0.33*\n) node(v2) {};
\draw (\n,-0.65*\n) node(v3) {};
\draw (\n,0.66*\n) node(v4) {};
\foreach \s in {2,...,4}
{
    \draw[arrow] (a) -- (v\s); \draw[arrow] (c) -- (v\s);
}

\node (b 8) at (0.83*\n, -0.30*\n) {};
\node (b 9) at (1.0*\n, -0.28*\n) {};
\node (b 10) at (1.17*\n, -0.30*\n) {};
\foreach \s in {8,...,10}
{
    \draw[arrow] (v3) -- (b \s);
}

\foreach \s in {1,...,5}
{
    \node[draw] (a \s) at ({360/10 * (\s + 3)}:\n/2) {};
    \draw[arrow] (a) -- (a \s); 
}

\draw (2.5*\n,0) node(c 1) {};
\draw (2.37*\n,0.4*\n) node(c 2) {};
\draw (2.4*\n,-0.3*\n) node(c 3) {};
\foreach \s in {1,...,3}
{
    \draw[arrow] (c) -- (c \s); 
}
\end{graph}
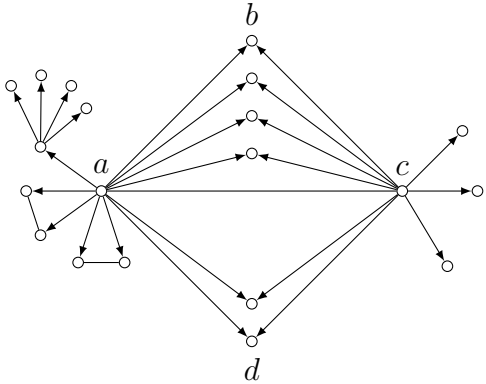
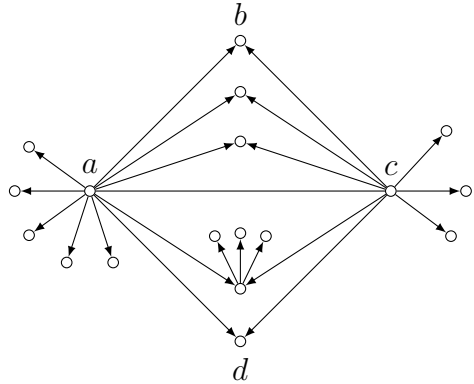
\caption{When the longest cycle is a $4$-cycle, use the orientation for the $P_6$-free case.}
\end{myfigure}
\begin{case}{3} Suppose a longest cycle $C$ in $T$ is a $4$-cycle.
This is exactly like the $4$-cycle case from the previous proof, except here one of $a$ or $c$ can have pendant triangles or pendant stars, or one of $v_1, \dots, v_m$ can have pendant edges.
Use the same partial orientation as before, and orient the pendant triangles, stars, and edges according to the standard orientation.
If there are at least three vertice $v_1, \dots, v_m$, then $T_- = \{a,c\}$ and $T_+ = \{v_1, \dots, v_m\}$.
Otherwise, $T_-$ and $T_+$ are empty.

\end{case}

\begin{case}{4} Suppose a longest cycle in $T$ is a $3$-cycle.
Then $T$ is made up entirely of pendant stars, pendant triangles, and pendant edges, except for at most one non-pendant star or edge.
Orient every pendant edge, triangle, and star according to the standard orientation, and leave the other (at most three) edges unoriented.

\begin{myfigure}
\begin{graph}
\def \n {2cm}
\draw (0,0) node(a) {};
\draw (\n/2,0.866*\n) node(b) {};
\draw (\n,0) node(c) {};
\draw (a) -- (b);
\draw (c) -- (b);
\draw (c) -- (a);

\foreach \s in {2,...,7}
{
    \node (a \s) at ({360/15 * (\s + 4)}:\n/2) {};
    \draw[arrow] (a) -- (a \s); 
}
\draw (a 2) -- (a 3);
\draw (a 4) -- (a 5);

\node (a 1) at ({360/15 * (4)}:\n/2) {};
\draw[arrow] (a) -- (a 1); 
\node (a 8) at (0, 0.85*\n) {};
\node (a 9) at (-0.2*\n, 0.8*\n) {};
\node (a 10) at (-0.35*\n, 0.65*\n) {};
\foreach \s in {8,...,10}
{
    \draw[arrow] (a 1) -- (a \s); 
}

\node (a 11) at ({360/15 * (13)}:\n/2) {};
\draw[arrow] (a) -- (a 11); 
\node (a 12) at (0.25*\n, -0.7*\n) {};
\node (a 13) at (0.45*\n, -0.7*\n) {};
\node (a 14) at (0.61*\n, -0.6*\n) {};
\node (a 15) at (0.67*\n, -0.43*\n) {};
\foreach \s in {12,...,15}
{
    \draw[arrow] (a 11) -- (a \s); 
}

\node (b 1) at (0.45*\n, 1.26*\n) {};
\node (b 2) at (0.63*\n, 1.23*\n) {};
\node (b 3) at (0.77*\n, 1.15*\n) {};
\foreach \s in {1,...,3}
{
    \draw[arrow] (b) -- (b \s); 
}

\node (c 1) at (1.23*\n, 0.27*\n) {};
\node (c 2) at (1.34*\n, 0.14*\n) {};
\node (c 3) at (1.36*\n, -0.05*\n) {};
\node (c 4) at (1.30*\n, -0.23*\n) {};
\foreach \s in {1,...,4}
{
    \draw[arrow] (c) -- (c \s); 
}
\end{graph}
\hspace{1.5cm}
\begin{graph}
\def \n {2cm}
\draw (0,0) node(a) {};
\draw (\n,0) node(b) {};
\draw (a) -- (b);

\foreach \s in {1,...,7}
{
    \node (a \s) at ({360/10 * (\s )}:\n/2) {};
    \draw[arrow] (a) -- (a \s); 
}
\draw (a 2) -- (a 3);
\draw (a 4) -- (a 5);

\node (a 8) at (0, -0.83*\n) {};
\node (a 9) at (-0.2*\n, -0.85*\n) {};
\node (a 10) at (-0.4*\n, -0.79*\n) {};
\foreach \s in {8,...,10}
{
    \draw[arrow] (a 7) -- (a \s); 
}

\node (a 11) at ({360/15 * (13)}:\n/2) {};
\draw[arrow] (a) -- (a 11); 
\node (a 12) at (0.25*\n, -0.7*\n) {};
\node (a 13) at (0.45*\n, -0.7*\n) {};
\node (a 14) at (0.61*\n, -0.6*\n) {};
\node (a 15) at (0.67*\n, -0.43*\n) {};
\foreach \s in {12,...,15}
{
    \draw[arrow] (a 11) -- (a \s); 
}

\node (b 2) at (0.7*\n, 0.3*\n) {};
\node (b 3) at (0.95*\n, 0.43*\n) {};
\node (b 4) at (1.2*\n, 0.4*\n) {};
\node (b 5) at (1.4*\n, 0.2*\n) {};
\node (b 6) at (1.4*\n, -0.07*\n) {};
\node (b 7) at (1.30*\n, -0.3*\n) {};
\node (b 8) at (1.0*\n, -0.4*\n) {};
\foreach \s in {2,...,8}
{
    \draw[arrow] (b) -- (b \s); 
}
\draw (b 4) -- (b 5); 
\draw (b 7) -- (b 8); 

\node (b 9) at (0.75*\n, 0.78*\n) {};
\node (b 10) at (0.95*\n, 0.83*\n) {};
\node (b 11) at (1.15*\n, 0.78*\n) {};
\foreach \s in {9,...,11}
{
    \draw[arrow] (b 3) -- (b \s); 
}

\end{graph}

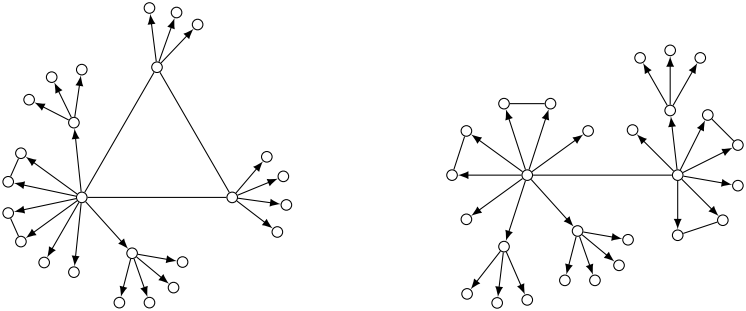
\caption{If the longest cycle in $T$ is a triangle, $T$ looks like one of these graphs.}
\end{myfigure}
\end{case}

\begin{case}{5} If $T$ has no cycles, choose a root vertex arbitrarily and orient every edge away from the root.
\end{case}
\end{proof}

It appears that all $P_N$-free graphs will admit a suitable partial orientation.
For any $N$, consider a two-colored complete graph on $R_2(P_N)-1 = N + \lfloor\frac{N}{2}\rfloor -2$ vertices with no monochromatic $P_N$ (i.e. a graph which gives the lower bound for $R_2(P_N)$).
We can construct such a graph as follows: take a red $K_{N-1}$ and a blue $K_{\lfloor \frac{N}{2}\rfloor - 1}$.
Color all of the edges in between these two parts blue.
We give a partial orientation to this graph in general, with $n=N + \lfloor \frac{N}{2} \rfloor - 2$, $t = N-1$, and $s = \lfloor\frac{N}{2}\rfloor -1$: leave the edges in the two monochromatic subgraphs unoriented and orient the remaining blue edges away from the blue $K_{\lfloor\frac{N}{2}\rfloor}$.

\begin{myfigure}
\begin{graph}
\def \n {1.5cm}
\foreach \s in {1,...,9}
{
    \node (a\s) at (1.0*\s*\n,0) {};
}
\foreach \s in {1,...,8}
{
    \pgfmathsetmacro{\z}{1+\s}
    \draw[red] (a\s) -- (a\z);
}
\foreach \s in {1,...,7}
{
    \pgfmathsetmacro{\z}{2+\s}
    \foreach \t in {\z,...,9}
    {
        \path (a\s) edge [red,out=340,in=200] (a\t);
    }
}

\foreach \s in {1,...,4}
{
    \node (b\s) at (1.25*\n + 1.5*\s*\n,1.5*\n) {};
    \foreach \t in {1,...,9}
    {
        \draw[bluearrow] (b\s) -- (a\t);
    }
}
\foreach \s in {1,...,3}
{
    \pgfmathsetmacro{\z}{1+\s}
    \draw[blue] (b\s) -- (b\z);
}
\foreach \s in {1,...,2}
{
    \pgfmathsetmacro{\z}{2+\s}
    \foreach \t in {\z,...,4}
    {
        \path (b\s) edge [blue,out=10,in=170] (b\t);
    }
}

\end{graph}
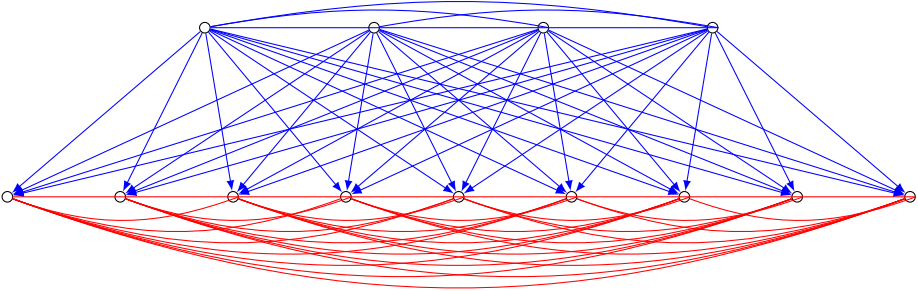
\caption{This graph proves $R_2(P_{10}) \geq 14$.}
\end{myfigure}
This motivates another open question:

\begin{question}
For every $N$, is the family of $P_N$-free graphs $(R_2(P_N)-1)$-bounded? 
In particular, does every $P_N$-free graph have an $(n,s,t)$-bounded orientation, for ${n=N + \lfloor \frac{N}{2} \rfloor - 2}$, $s = \lfloor\frac{N}{2}\rfloor -1$, and $t = N-1$?
\end{question}


\begin{thebibliography}{99}

\bibitem{BGY} A. Bialostocki, A. Gy\'arf\'as, Replacing the host $K_n$ by $n$-chromatic graphs in Ramsey-type results, manuscript.

\bibitem{CL} E.J. Cockayne, P.J. Lorimer, The Ramsey number for stripes, {\em J. Australian Math. Soc. series A.} {\bf 19} (1975), 252-256.

\bibitem{GeGy} L. Gerencs\'er and A. Gy\'arf\'as, On Ramsey-Type Problems, {\em Annales Universitatis Scientiarum Budapestinensis, E\"otv\"os Sect. Math.}, {\bf 10} (1967) 167-170.

\bibitem{ErGa} P. Erd\H{o}s, T. Gallai, On maximal paths and circuits of graphs, {\em Acta Math. Acad. Sci. Hung.}, {\bf 10} (1959) 337-356.

\bibitem{Ir} R.W. Irving, Generalised Ramsey Numbers for Small Graphs, {\em Discrete Mathematics}, {\bf 9} (1974) 251-264.

\bibitem{BuRo} S.A. Burr and J.A. Roberts, On Ramsey Numbers for Stars, {\em Utilitas Mathematica}, {\bf 4} (1973) 217-220.

\bibitem{Wall} W.D. Wallis, On a Ramsey Number for Paths, {\em Journal of Combinatorics, Information \& System Sciences}, {\bf 6} (1981) 295-296.
\end{thebibliography}
\end{document}